\documentclass[a4paper, 12pt]{amsart}

\usepackage{amsmath,amsthm,amssymb}

\usepackage{color}
\usepackage{ulem}

\theoremstyle{plain}
	\newtheorem{thm}{Theorem}[section]
	\newtheorem{cor}[thm]{Corollary}
	\newtheorem{lem}[thm]{Lemma}

\theoremstyle{definition}

\theoremstyle{remark}

\makeatletter
	
	\@addtoreset{equation}{section}
\makeatother

\def\C{\mathbb{C}}

\def\B{\mathcal{B}}

\def\M{\mathcal{M}}

\def\a{\alpha}

\def\d{\delta}

\def\Lam{\Lambda}
\def\lam{\lambda}

\def\imp{\Rightarrow}

\def\s{\setminus}

\def\ds{\displaystyle}

\let\Ran\relax
	\DeclareMathOperator{\Ran}{Ran}
\let\supp\relax
	\DeclareMathOperator{\supp}{supp}

\newcommand{\paren}[1]{\left(#1\right)}

\begin{document}

\title[Range preserving maps]{Range preserving maps between the spaces of continuous functions with values in a locally convex space}
\author[Y. Enami]{Yuta Enami}
\address{Department of Mathematics, Faculty of Science, Niigata University, Niigata 950-2181 Japan}
\email{enami@m.sc.niigata-u.ac.jp}

\maketitle

\begin{abstract}
Let $C(X,E)$ be the linear space of all continuous functions on a compact \mbox{Hausdorff} space $X$ with values in a locally convex space $E$. We characterize maps $T:C(X,E)\to C(Y,E)$ which satisfy $\Ran(TF-TG)\subset\Ran(F-G)$ for all $F,G\in C(X,E)$.
\end{abstract}

\section{Introduction}

Let $X$ and $Y$ be compact \mbox{Hausdorff} spaces, and let $C(X)$ and $C(Y)$ be the \mbox{Banach} algebras of all continuous complex-valued functions on $X$ and $Y$, respectively. The range of $f\in C(X)$ will be denoted by $\Ran(f)$. For each $f\in C(X)$, the range $\Ran(f)$ coincides with the spectrum $\sigma(f)$ of $f$. 

\mbox{Gleason} \cite{Gle} and \mbox{Kahane} and \mbox{\.{Z}elazko} \cite{KZ} have proved independently the following theorem known as the \mbox{Gleason}-\mbox{Kahane}-\mbox{\.{Z}elazko} theorem: Every linear functional $T$ on a complex \mbox{Banach} algebra $A$ which satisfies $Ta\in \sigma(a)$ for all $a\in A$ is multiplicative, where $\sigma(a)$ is the spectrum of $a$. It follows that every linear map $T:C(X)\to C(Y)$ which satisfies $\Ran(Tf)\subset \Ran(f)$ for all $f\in C(X)$ is multiplicative. \mbox{Kowalski} and \mbox{S{\l}odkowski} \cite{KS} have proved a surprising generalization of the \mbox{Gleason}-\mbox{Kahane}-\mbox{\.{Z}elazko} theorem as follows: Every map $T:A\to\C$, not necessarily linear, which satisfies $T0=0$ and $Ta-Tb\in\sigma(a-b)$ for all $a,b\in A$ is linear and multiplicative. Applying the \mbox{Kowalski}-\mbox{S{\l}odkowski} theorem, \mbox{Hatori}, \mbox{Miura} and \mbox{Takagi} \cite{HMT} have proved that every map $T$ from a unital \mbox{Banach} algebra $A$ into a unital semi-simple commutative \mbox{Banach} algebra $B$ which satisfies $\sigma(Ta+Tb)\subset\sigma(a+b)$ for all $a,b\in A$ is linear and multiplicative.

\mbox{Ghodrat} and \mbox{Sady} \cite{GS} introduced a spectrum for \mbox{Banach} module elements. Let $C(X,E)$ be a Banach $C(X)$-module of all continuous functions on $X$ with values in a \mbox{Banach} space $E$. For each $F\in C(X,E)$, the spectrum of $F$ as a \mbox{Banach} module element is a subset of $\{\Lam(F(x)):x\in X,\Lam\in (E^\ast)_1\s\{0\}\}$, where $(E^*)_1$ is the closed unit ball of the dual space of $E$. Moreover, under some assumptions, they have characterized spectrum preserving bounded linear maps between unital \mbox{Banach} modules over a unital \mbox{Banach} algebra.  

In this paper, We will consider maps $T:C(X,E)\to C(Y,E)$ which satisfy $\Ran(TF-TG)\subset \Ran(F-G)$ for all $F,G\in C(X,E)$, where $\Ran(F)=\{F(x):x\in X\}$. The following are main results of this paper.

\begin{thm}\label{MT}
Let $X$ and $Y$ be compact \mbox{Hausdorff} spaces, and let $E$ be a locally convex space which is not $\{0_E\}$. If $T:C(X,E)\to C(Y,E)$ is a map, not necessarily linear, which satisfies
$$
\Ran(TF-TG)\subset \Ran(F-G)
$$
for every $F,G\in C(X,E)$, then there exists a continuous map $\varphi:Y\to X$ such that 
$$
TF=T(1\otimes 0_E)+F\circ\varphi
$$
for every $F\in C(X,E)$. In particular, if, in addition, $T(1\otimes 0_E)=1\otimes 0_E$, then $T$ is linear.
\end{thm}

\begin{cor}\label{MTC} Let $T$ and $\varphi$ be as in Theorem \ref{MT}. Statements {\rm (a)} through {\rm (c)} are equivalent.

\noindent{\rm (a)} $T$ is injective.

\noindent{\rm (b)} $\varphi$ is surjective.

\noindent{\rm (c)} $\Ran(TF-TG)=\Ran(F-G)$ for every $F,G\in C(X,E)$.

\noindent{}Also, statements {\rm (d)} and {\rm (e)} are equivalent.

\noindent{\rm (d)} $T$ is surjective.

\noindent{\rm (e)} $\varphi$ is injective.
\end{cor}


\section{Preliminaries}

We denote by $\sigma(a)$ the spectrum of an element $a$ in a \mbox{Banach} algebra $A$. The following theorem was proved by \mbox{Kowalski} and \mbox{S{\l}odkowski}.

\begin{lem}[{\cite[Theorem 1.2]{KS}}]\label{KS}
Let $A$ be a \mbox{Banach} algebra. Let $\d:A\to\C$ satisfy

\noindent{\rm (i)} $\d(0)=0$, and

\noindent{\rm (ii)} $\d(a)-\d(b)\in\sigma(a-b)$ for every $a,b\in A$.

\noindent{}Then $\d$ is linear and multiplicative.
\end{lem}

Let $A$ be a unital commutative \mbox{Banach} algebra. The set of all non-zero multiplicative linear functionals on $A$ is denoted by $M_A$. It is well-known that $M_A$ is a weak star-closed subset of the unit ball of dual $A^\ast$ of $A$, and hence $M_A$ is a compact \mbox{Hausdorff} space with respect to the relative weak star-topology. The compact \mbox{Hausdorff} space $M_A$ is called the {\it maximal ideal space} of $A$.

Let $X$ be a non-empty compact \mbox{Hausdorff} space, and let $C(X)$ be the \mbox{Banach} algebra of all complex-valued continuous functions on $X$ with the supremum norm. Every $x\in X$ determines the point evaluation $\d_x\in \M_{C(X)}$, defined by
$$
\d_x(f)=f(x)\qquad (f\in C(X)).
$$
Conversely, every non-zero multiplicative linear functional on $C(X)$ is $\d_x$ for some $x\in X$. Moreover, the canonical map $X\ni x\mapsto\d_x\in M_{C(X)}$ is a homeomorphism. 

A {\it topological vector space} is a complex vector space $E$ endowed with a \mbox{Hausdorff} topology with respect to which addition $E\times E\ni (u,v)\mapsto u+v\in E$ and multiplication by scalar $\C\times E\ni (\lam,u)\mapsto \lam u\in E$ are continuous. We denote by $0_E$ the origin of $E$. A {\it local basis} of $E$ is a basis of neighborhood of $0_E$. A subset $B\subset E$ is said to be {\it balanced} if 
$$
\lam B\subset B
$$
for all $\lam\in\C$ with $|\lam|\le 1$. 

A {\it locally convex space} is a topological vector space which has a local basis consisting of convex neighborhoods. The following lemma is well-known, so we omit its proof. 

\begin{lem}[{\cite[Theorem 1.14]{Rud}}]\label{LemBN}
Every locally convex space has a local basis which consists of balanced convex neighborhoods.
\end{lem}

In order to define a topology on the space of continuous functions with values in a locally convex space, we will use the following lemma in the next section. We omit its proof since it is routine.

\begin{lem}\label{LemVT}
Let $L$ be a complex vector space, and let $\B$ be a family of non-empty subsets of $L$ which satisfies the following properties.

\noindent{\rm (i)} For each pair $B_1,B_2\in\B$ there exists $B_3\in\B$ such that $B_3\subset B_1\cap B_2$.

\noindent{\rm (ii)} For each $B_1\in\B$ there exists $B_2\in\B$ such that $B_2+B_2\subset B_1$.

\noindent{\rm (iii)} Each $B\in\B$ is balanced.

\noindent{\rm (iv)} For each $u\in L$ and $B\in\B$ there exists $t>0$ such that $u\in tB$.

\noindent{\rm (v)} $\ds{\bigcap_{B\in\B}B=\{0_{L}\}}$.

\noindent{}Then there is a unique topology on $L$ for which $L$ is a topological vector space and $\B$ is a local basis of $L$.
\end{lem}


\section{The space $C(X,E)$}

Throughout this section, let $X$ be a non-empty compact \mbox{Hausdorff} space, let $E$ be a locally convex space which is not $\{0_E\}$, and let $\B_E$ be a fixed local basis of $E$ which consists of balanced convex neighborhoods.

The set of all continuous functions on $X$ with values in $E$ is denoted by $C(X,E)$. For $F\in C(X,E)$, the {\it range} of $F$ is denoted by $\Ran(F)$:
$$
\Ran(F)=\{F(x)\in E:x\in X\}.
$$
With pointwise vector operations, $C(X,E)$ is a complex vector space. For each $B\in\B_E$, let
$$
V_X(B)=\{F\in C(X,E):\Ran(F)\subset B\}.
$$
Note that $V_X(B)$ is convex for all $B\in\B_E$. Indeed, if $F,G\in V_X(B)$, and if $0<\lam<1$, then 
$$
\Ran((1-\lam)F+\lam G)\subset B
$$
by the convexity of $B$, and hence $(1-\lam)F-\lam G\in V_X(B)$. Now the family $\{V_X(B):B\in\B_E\}$ satisfies the properties (i) through (v) in Lemma \ref{LemVT}. Hence there exists a unique topology on $C(X,E)$ such that $C(X,E)$ is a locally convex space and the family $\{V_X(B):B\in\B_E\}$ is a local basis of $C(X,E)$. This topology is called the {\it topology of uniform convergence}. 

For $f\in C(X)$ and $u\in E$, define $f\otimes u:X\to E$ by
$$
(f\otimes u)(x)=f(x)u
$$
for each $x\in X$. Then $f\otimes u\in C(X,E)$. The subspace of $C(X,E)$ spanned by the functions of the form $f\otimes u$ will be denoted by $C(X)\otimes E$, and is called the {\it algebraic tensor product} of $C(X)$ and $E$. Note that for each $u\in E$, the function $1\otimes u$ is the constant function on $X$ whose value is only $u$ at each point of $X$. Note also that for each $f\in C(X)$ and $u\in E$, we have $f\otimes 0_E=0\otimes u=1\otimes 0_E$.

\begin{lem}\label{LemDS}
The algebraic tensor product $C(X)\otimes E$ is dense in $C(X,E)$ with respect to the topology of uniform convergence.
\end{lem}

\begin{proof} To see this, fix $F\in C(X,E)$ and $B\in\B_E$. For each $x_0\in X$, set
$$
W_{x_0}=\{x\in X:F(x)-F(x_0)\in B\}.
$$
The continuity of $F$ implies that $W_{x_0}$ is a neighborhood of $x_0$. Since $X$ is a compact space, we can choose finitely many points $x_1,\dots, x_n\in X$ such that 
$$
X=\bigcup_{j=1}^n W_{x_j}.
$$
Let $\{h_1,\dots, h_n\}$ be a partition of unity subordinate to the finite cover $\{W_{x_1},\dots, W_{x_n}\}$, that is, $h_1,\dots, h_n$ are continuous real-valued functions on $X$ such that $h_j\ge0$, $\supp(h_j)\subset W_{x_j}$, where $\supp(h_j)$ is the closure of $\{x\in X:h_j(x)\ne0\}$, and $h_1+\dots+h_n=1$. Let 
$$
G=\sum_{j=1}^n h_j\otimes F(x_j)\in C(X)\otimes E.
$$
Note that if $x\in X$ and $h_j(x)\ne 0$, then $F(x)-F(x_j)\in B$, since $\supp(h_j)\subset W_{x_j}$. For each $x\in X$, the convexity of $B$ implies that 
\begin{align*}
F(x)-G(x)&=\sum_{j=1}^nh_j(x)F(x)-\sum_{j=1}^nh_j(x)F(x_j)\\
&=\sum_{j=1}^nh_j(x)(F(x)-F(x_j))\in B.
\end{align*}
Hence we have $F-G\in V_X(B)$. This proves that $C(X)\otimes E$ is dense in $C(X,E)$ with respect to the topology of uniform convergence.
\end{proof}


\section{Proof of the main theorem}

In this section, we will prove main results. Throughout this section, let $X$ and $Y$ be non-empty compact Hausdorff spaces, let $E$ be a locally convex space which is not $\{0_E\}$, and let $\B_E$ be a fixed local basis of $E$ which consists of balanced convex neighborhoods. 

Assume that $T:C(X,E)\to C(Y,E)$ is a map which satisfies 
$$
\Ran(TF-TG)\subset\Ran(F-G)
$$
for every $F,G\in C(X,E)$. Without loss of generality, we can assume that $T(1\otimes 0_E)=1\otimes 0_E$. Then we have
\begin{align}\label{RP0}
\Ran(TF)\subset\Ran(F)
\end{align}
for all $F\in C(X,E)$. 

Fix $u\in E\s\{0_E\}$. For each $f \in C(X)$ and $y \in Y$,
$$
T(f\otimes u)(y)\in\Ran(T(f\otimes u))\subset \Ran(f\otimes u)=(\Ran(f))u.
$$
Then there is a unique complex number $(\tilde{T}_uf)(y)\in\Ran(f)$ such that
\begin{align}\label{k}
T(f\otimes u)(y)=(\tilde{T}_uf)(y)u.
\end{align}
The complex number $\tilde{T}_uf(y)$ determines a function $\tilde{T}_uf$ on $Y$ with \eqref{k} for each $f \in C(X)$ and $u \in E \s\{0_E\}$.

\begin{lem}\label{C}
The function $\tilde{T}_uf$ is continuous on $Y$ for every $f\in C(X)$ and $u\in E\s\{0_E\}$.
\end{lem}

\begin{proof}
Fix $y\in Y$, and choose a net $\{y_\a\}_\a$ in $Y$ which converges to $y$. Since the function $T(f\otimes u)$ is continuous, we have
$$
\tilde{T}_uf(y_\a)u=T(f\otimes u)(y_a)\to T(f\otimes u)(y)=\tilde{T}_uf(y)u.
$$
Since $u\ne 0_E$, the net $\{\tilde{T}_uf(y_\a)\}_\a$ converges to $\tilde{T}_uf(y)$. This proves the lemma.
\end{proof}

By Lemma \ref{C}, the correspondence $f\mapsto \tilde{T}_uf$ is a well-defined map from $C(X)$ into $C(Y)$ for each $u\in E\setminus\{0_E\}$, and we denote it by $\tilde{T}_u:C(X)\to C(Y)$. Since $\tilde{T}_uf\in C(Y)$, we can rewrite equality \eqref{k} as follows:
\begin{align}\label{k1}
T(f\otimes u)=(\tilde{T}_uf)\otimes u \qquad (f \in C(X), u\in E\s\{ 0_E \}).
\end{align}

\begin{lem}\label{RPP}
For each $u\in E\s\{0_E\}$, the map $\tilde{T}_u$ satisfies $\tilde{T}_u0=0$ and
$$
\Ran(\tilde{T}_uf-\tilde{T}_ug)\subset \Ran(f-g)
$$
for every $f,g\in C(X)$.
\end{lem}

\begin{proof}
If $y\in Y$, then 
$$
(\tilde{T}_u0)(y)u=T(0\otimes u)(y)=T(1\otimes0_E)(y)=0_E,
$$
and hence $(\tilde{T}_u0)(y)=0$. This proves that $\tilde{T}_u0=0$. 

Let $f,g\in C(X)$, and let $y\in Y$. By \eqref{k1}, 
$$
(\tilde{T}_uf-\tilde{T}_ug)(y)u=T(f\otimes u)(y)-T(g\otimes u)(y).
$$
Hence $(\tilde{T}_uf-\tilde{T}_ug)(y)u \in \Ran(T(f\otimes u)-T(g\otimes u))
\subset \Ran(f\otimes u-g\otimes u)$.
Because $\Ran(f\otimes u-g\otimes u) = \Ran(f-g)u$, we have that
$(\tilde{T}_uf-\tilde{T}_ug)(y)u$ is in $\Ran(f-g)u$, which shows
$\Ran(\tilde{T}_uf-\tilde{T}_ug)\subset \Ran(f-g)$.
\end{proof}

By a similar argument to \cite[Theorem 3.1.]{HMT}, we can prove the following lemma. For the sake of completeness, we give its proof.

\begin{lem}\label{TT}
For each $u\in E\s\{0_E\}$, there exists a continuous map $\varphi_u:Y\to X$ such that
$$
\tilde{T}_uf=f\circ\varphi_u
$$
for all $f\in C(X)$. In particular, the map $\tilde{T}_u$ is linear.
\end{lem}

\begin{proof}
Fix $y\in Y$. We shall prove that the map $\d_y\circ\tilde{T}_u$ is a non-zero multiplicative linear functional on $C(X)$. By Lemma \ref{RPP}, $\tilde{T}_u0=0$. This implies that $(\d_y\circ \tilde{T}_u)(0)=0$. Also, the map $\tilde{T}_u$ satisfies $\Ran(\tilde{T}_uf-\tilde{T}_ug)\subset \Ran(f-g)$ for every $f,g\in C(X)$, and hence
\begin{align*}
(\d_y\circ \tilde{T}_u)(f)-(\d_y\circ \tilde{T}_u)(g)&=(\tilde{T}_uf-\tilde{T}_ug)(y)\\
&\in\Ran(\tilde{T}_uf-\tilde{T}_ug)\subset\Ran(f-g),
\end{align*}
for every $f,g\in C(X)$. Lemma \ref{KS}, the \mbox{Kowalski}-\mbox{S{\l}odkowski} theorem, shows that
$\d_y\circ \tilde{T}_u$ is linear and multiplicative. Since $\tilde{T}_u0 = 0$,
$$
\Ran(\tilde{T}_u1)=\Ran(\tilde{T}_u1-\tilde{T}_u0)\subset \Ran(1)=\{1\},
$$
and thus $(\d_y\circ \tilde{T}_u)(1)=1$. This shows that the functional $\d_y\circ \tilde{T}_u$ on $C(X)$ is non-zero. Hence there is a unique point $\varphi_u(y)\in X$ such that 
$$
\d_y\circ \tilde{T}_u=\d_{\varphi_u(y)}.
$$
Now we have a map $\varphi_u:Y\to X$ such that
$$
(\tilde{T}_uf)(y)=f(\varphi_u(y))
$$
for every $f \in C(X)$.

It suffices to prove that the map $\varphi_u$ is continuous. To see this, fix $y\in Y$, and choose a net $\{y_\a\}_\a$ in $Y$ which converges to $y$. For every $f\in C(X)$, the continuity of $\tilde{T}_uf$ implies that
$$
\d_{\varphi_u(y_\a)}(f)=(\tilde{T}_uf)(y_\a)\to(\tilde{T}_uf)(y)=\d_{\varphi_u(y)}(f),
$$
and hence $\d_{\varphi_u(y_\a)}\to\d_{\varphi_u(y)}$ with respect to the relative weak star-topology. Hence we obtain $\varphi_u(y_\a)\to\varphi_u(y)$, and this proves the continuity of $\varphi_u$.
\end{proof}

\begin{lem}\label{add}
For each $f,g\in C(X)$ and $u,v\in E$,
\begin{align}\label{addeq}
T(f\otimes u+g\otimes v)=T(f\otimes u)+T(g\otimes v).
\end{align}
\end{lem}

\begin{proof}
Since $T(1\otimes 0_E)=1\otimes 0_E$, it suffices to prove that equality \eqref{addeq} holds when both $u$ and $v$ are not $0_E$. Let $u,v\in E\s\{0_E\}$, and let $f,g\in C(X)$.

First, suppose that $u$ and $v$ are linearly dependent, that is, $v=\a u$ for some non-zero complex number $\a$. Using the linearity of $\tilde{T}_u$, we have
\begin{align*}
T(f\otimes u+g\otimes v)&=T((f+\a g)\otimes u)=\tilde{T}_u(f+\a g)\otimes u \\
&=(\tilde{T}_uf+\tilde{T}_u(\a g))\otimes u =\tilde{T}_uf\otimes u+\tilde{T}_u(\a g)\otimes u \\
&=T(f\otimes u)+T(g\otimes v).
\end{align*}
Hence, we obtain equality \eqref{addeq} if $u$ and $v$ are linearly dependent.

Second, suppose that $u$ and $v$ are linearly independent. Fix $y\in Y$. By equality \eqref{RP0}, we have
\begin{align*}
T(f\otimes u+g\otimes v)(y)&\in \Ran(T(f\otimes u+g\otimes v))\\
&\subset\Ran(f\otimes u+g\otimes v).
\end{align*}
Hence there exists a point $x_0\in X$ such that 
\begin{align}\label{a1}
T(f\otimes u+g\otimes v)(y)=f(x_0)u+g(x_0)v.
\end{align}
On the other hand, $T(f\otimes u+g\otimes v)(y)-T(f\otimes u)(y)$ is in $\Ran(T(f\otimes u+g\otimes v)-T(f\otimes u)) \subset\Ran(g\otimes v)$, and then there exists a point $x_1\in X$ such that 
\begin{align}\label{a2}
T(f\otimes u+g\otimes v)(y)-T(f\otimes u)(y)=g(x_1)v.
\end{align}
By a similar argument, there exists a point $x_2\in X$ such that 
\begin{align}\label{a3}
T(f\otimes u+g\otimes v)(y)-T(g\otimes v)(y)=f(x_2)u.
\end{align}
It follows from \eqref{a1} and \eqref{a2} that
$$
f(x_0)u-T(f\otimes u)(y)=g(x_1)v-g(x_0)v.
$$
Since $T(f\otimes u)(y)=
\tilde{T}_uf(y)u$,
the linear independence of $u$ and $v$ now implies 
\begin{align}\label{a4}
T(f\otimes u)(y)=f(x_0)u.
\end{align}
From \eqref{a1} and \eqref{a3}, we conclude similarly that
\begin{align}\label{a5}
T(g\otimes v)(y)=g(x_0)v.
\end{align}
Equalities \eqref{a1}, \eqref{a4}, and \eqref{a5} yield \eqref{addeq}.
\end{proof}

\begin{lem}\label{TuTv}
For every $u,v\in E\s\{0_E\}$, $\tilde{T}_u=\tilde{T}_v$.
\end{lem}

\begin{proof}
Let $u,v\in E\s\{0_E\}$. First, suppose that $u$ and $v$ are linearly dependent, that is, $v=\a u$ for some non-zero complex number $\a$. If $f\in C(X)$, then
\begin{align*}
(\tilde{T}_uf)\otimes v&=(\tilde{T}_uf)\otimes (\a u)=\a(\tilde{T}_uf)\otimes u\\
&=(\tilde{T}_u(\a f))\otimes u=T((\a f)\otimes u)\\
&=T(f\otimes v)=(\tilde{T}_vf)\otimes v,
\end{align*}
and hence $\tilde{T}_uf=\tilde{T}_vf$. This proves that $\tilde{T}_u=\tilde{T}_v$ if $u$ and $v$ are linearly dependent.

Second, suppose that $u$ and $v$ are linearly independent. Let $f\in C(X)$. From Lemma \ref{add}, we have 
\begin{align*}
(\tilde{T}_uf)\otimes u+(\tilde{T}_vf)\otimes v&=T(f\otimes u)+T(g\otimes v)=T(f\otimes u+f\otimes v)\\
&=(\tilde{T}_{u+v}f)\otimes (u+v)=(\tilde{T}_{u+v}f)\otimes u+(\tilde{T}_{u+v}f)\otimes v,
\end{align*}
and hence the linear independence of $u$ and $v$ implies that $\tilde{T}_uf=\tilde{T}_{u+v}f=\tilde{T}_vf$. This proves that $\tilde{T}_u=\tilde{T}_v$.
\end{proof}

\begin{lem}\label{Conti}
The map 
$T:C(X,E) \to C(Y,E)$ is continuous with the topology of uniform convergence.
\end{lem}

\begin{proof}
We shall prove the continuity of $T$ at each $F\in C(X,E)$. Fix $B\in \B_E$. Let $\{F_\a\}_{\a\in\Lam}$ be a net which converges to $F$, where $\Lam=(\Lam,\preceq)$ is a directed set. Then there exists $\a_0\in\Lam$ such that $F_\a-F\in V_X(B)$ for all $\a\succeq \a_0$. By the definition of the set $V_X(B)$,
$$
\Ran(TF_\a-TF)\subset\Ran(F_\a-F)\subset B
$$
for all $\a\succeq\a_0$, and hence $TF_\a-TF\in V_Y(B)$. This proves that the net $\{TF_\a\}_{\a\in\Lam}$ converges to $TF$.
\end{proof}

\begin{proof}[Proof of Theorem \ref{MT}]
By Lemma \ref{TuTv}, $\varphi_u=\varphi_v$ for all $u,v\in E\s\{0_E\}$. Hence we may write $\varphi$ instead of $\varphi_u$. Equality \eqref{k1} with Lemma \ref{TT} shows that
\begin{align}\label{k25}
T(f\otimes u)=(f\otimes u)\circ\varphi
\end{align}
for all $f\in C(X)$ and $u\in E\s\{0_E\}$. From Lemma \ref{add} and equality \eqref{k25}, we see that $T$ is a composition operator with symbol $\varphi$ on the algebraic tensor product $C(X)\otimes E$:
\begin{align}\label{k3}
T\paren{\sum_{j=1}^nf_j\otimes u_j}=\paren{\sum_{j=1}^nf_j\otimes u_j}\circ\varphi
\end{align}
for all $f_1,\dots,f_n\in C(X)$ and $u_1,\dots, u_n\in E\s\{0_E\}$. 

By equality \eqref{k3}, Lemma \ref{LemDS}, and Lemma \ref{Conti}, $T$ is a continuous map from $C(X,E)$ into the Hausdorff space $C(Y,E)$ which coincides with the composition operator with symbol $\varphi$ on the dense subspace $C(X)\otimes E$. Hence $T$ must coincide with composition operator with symbol $\varphi$ on the whole space $C(X,E)$:
$$
TF=F\circ\varphi
$$
for all $F\in C(X,E)$, as desired.
\end{proof}

\begin{proof}[Proof of Corollary \ref{MTC}]
\noindent (a)$\imp$(b) Suppose that $\varphi$ is not surjective. Then we can choose a point $x_0\in X\s\varphi(Y)$. The continuity of $\varphi$ implies that $\varphi(Y)$ is compact in $X$. By \mbox{Urysohn}'s lemma, there exists a function $f\in C(X)$ such that $f(x_0)=1$ and $f(x)=0$ for all $x\in\varphi(Y)$. Let $u\in E\s\{0_E\}$, and define $F=f\otimes u$. Then $F\ne1\otimes 0_E$ while $TF=(f \circ \varphi)\otimes u=0_E$. Hence $T$ cannot be injective. \\

\noindent (b)$\imp$(c) Suppose that $\varphi$ is surjective. Let $F,G\in C(X,E)$. It suffices to show that $\Ran(F-G)\subset \Ran(TF-TG)$. If $x\in X$, then there is a point $y\in Y$ such that $x=\varphi(y)$, and hence
\begin{align*}
(F-G)(x)&=F(\varphi(y))-G(\varphi(y))\\
&=(TF-TG)(y)\in\Ran(TF-TG).
\end{align*}
This proves that $\Ran(F-G)\subset \Ran(TF-TG)$.
\\

\noindent (c)$\imp$(a) Suppose that the statement (c) holds. If $TF=TG$, then 
$$
\Ran(F-G)=\Ran(TF-TG)=\{0_{E}\},
$$
and hence $F=G$. 
\\

\noindent (d)$\imp$(e) Suppose that $\varphi$ is not injective. Then there are distinct points $y_1$ and $y_2$ in $Y$ such that $\varphi(y_1)=\varphi(y_2)$. It follows that $TF(y_1)=TF(y_2)$ for all $F\in C(X,E)$. On the other hand, there exits $H\in C(Y,E)$ such that $H(y_1)\ne H(y_2)$. For instance, we can choose $H$ to be $h\otimes u$ for $h\in C(Y)$ which satisfies $h(y_1)\ne h(y_2)$ and for $u\in E\s\{0_E\}$. Hence $T$ cannot be surjective.
\\

\noindent (e)$\imp$(d) Note that if the map $\tilde{T}:C(X)\to C(Y)$ defined by 
$$
\tilde{T}f=f\circ\varphi\,\,\, (f\in C(X))
$$
is surjective, then $T$ is also surjective. Indeed, if $\tilde{T}$ is surjective, then $T$ carries $C(X)\otimes E$ onto $C(Y)\otimes E$, and hence $T$ must be surjective.

Suppose that $\varphi$ is injective. It suffices to show that $\tilde{T}$ is surjective. Let $h\in C(Y)$. Since $\varphi$ is a homeomorphism from $Y$ onto its image $\varphi(Y)$, there exists a continuous map $\psi:\varphi(Y)\to Y$ such that the composition $\psi\circ\varphi$ is identity on $Y$. Now $h\circ\psi$ is a function defined on the closed subset $\varphi(Y)$ of compact \mbox{Hausdorff} space $X$. By the \mbox{Tietze} extension theorem, there exists a function $f\in C(X)$ such that $f=h\circ\psi$ on $\varphi(Y)$. Hence $\tilde{T}f=(h\circ \psi)\circ \varphi=h$. This proves that $\tilde{T}$ is surjective.
\end{proof}

\end{document}